\documentclass[12pt]{article}
\usepackage{tikz}
\usepackage{amsmath}
\usepackage{indentfirst, latexsym, bm, amssymb}

\begin{document}

\newtheorem{theorem}{Theorem}[section]
\newtheorem{corollary}[theorem]{Corollary}
\newtheorem{definition}[theorem]{Definition}
\newtheorem{conjecture}[theorem]{Conjecture}
\newtheorem{question}[theorem]{Question}
\newtheorem{lemma}[theorem]{Lemma}
\newtheorem{proposition}[theorem]{Proposition}
\newtheorem{example}[theorem]{Example}
\newenvironment{proof}{\noindent {\bf
Proof.}}{\rule{3mm}{3mm}\par\medskip}
\newcommand{\remark}{\medskip\par\noindent {\bf Remark.~~}}
\newcommand{\pp}{{\it p.}}
\newcommand{\de}{\em}

\newcommand{\JEC}{{\it Europ. J. Combinatorics},  }
\newcommand{\JCTB}{{\it J. Combin. Theory Ser. B.}, }
\newcommand{\JCT}{{\it J. Combin. Theory}, }
\newcommand{\JGT}{{\it J. Graph Theory}, }
\newcommand{\ComHung}{{\it Combinatorica}, }
\newcommand{\DM}{{\it Discrete Math.}, }
\newcommand{\ARS}{{\it Ars Combin.}, }
\newcommand{\SIAMDM}{{\it SIAM J. Discrete Math.}, }
\newcommand{\SIAMADM}{{\it SIAM J. Algebraic Discrete Methods}, }
\newcommand{\SIAMC}{{\it SIAM J. Comput.}, }
\newcommand{\ConAMS}{{\it Contemp. Math. AMS}, }
\newcommand{\TransAMS}{{\it Trans. Amer. Math. Soc.}, }
\newcommand{\AnDM}{{\it Ann. Discrete Math.}, }
\newcommand{\NBS}{{\it J. Res. Nat. Bur. Standards} {\rm B}, }
\newcommand{\ConNum}{{\it Congr. Numer.}, }
\newcommand{\CJM}{{\it Canad. J. Math.}, }
\newcommand{\JLMS}{{\it J. London Math. Soc.}, }
\newcommand{\PLMS}{{\it Proc. London Math. Soc.}, }
\newcommand{\PAMS}{{\it Proc. Amer. Math. Soc.}, }
\newcommand{\JCMCC}{{\it J. Combin. Math. Combin. Comput.}, }
\newcommand{\GC}{{\it Graphs Combin.}, }

\title{ Laplacian coefficients of unicyclic graphs with the number of leaves and girth
\thanks{
This work is supported by National Natural Science Foundation of
China (No.11271256), innovation Program of Shanghai Municipal Education Commission (No.14ZZ016), the Ph.D. Programs Foundation of Ministry of Education of China (No.20130073110075) and
Funds of Innovation of Shanghai Jiao Tong University (No.Z071007).  }}
\author{Jie Zhang, Xiao-Dong Zhang\thanks{Corresponding  author ({\it E-mail address:}
xiaodong@sjtu.edu.cn)}
\\
{\small Department of Mathematics, and MOE-LSC,}\\
{\small Shanghai Jiao Tong University} \\
{\small  800 Dongchuan road, Shanghai, 200240,  PR China}\\
 }
\date{}
\maketitle
 \begin{abstract}
   Let $G$ be a graph of order $n$ and let $\mathcal{L}(G,\lambda)=\sum_{k=0}^n
(-1)^{k}c_{k}(G)\lambda^{n-k}$ be the characteristic polynomial of its Laplacian
matrix. Motivated by Ili\'{c} and Ili\'{c}'s conjecture
 [A.~Ili\'{c}, M.~Ili\'{c}, Laplacian coefficients of trees with given number of leaves or vertices of degree two,
 Linear Algebra and its Applications 431(2009)2195-2202.]
 on  all extremal graphs which minimize all the Laplacian coefficients in the set $\mathcal{U}_{n,l}$
  of all $n$-vertex unicyclic graphs with the number of leaves $l$, we investigate properties of
  the minimal elements in the partial set $(\mathcal{U}_{n,l}^g, \preceq)$ of the Laplacian coefficients,     where $\mathcal{U}_{n,l}^g$ denote the set of $n$-vertex unicyclic graphs with the
   number of leaves $l$ and  girth $g$. These results are used to disprove  their conjecture. Moreover, the graphs with minimum Laplacian-like energy
    in $\mathcal{U}_{n,l}^g$ are also studied.

 \end{abstract}


\vskip 0.5cm

\section{Introduction}
Let $G=(V,E)$ be a simple graph with $n=|V|$ vertices and
$n$ edges and $L(G)=D(G)-A(G)$ be its {\it Laplacian matrix,} where $A(G)$ and $D(G)$ are
its adjacency and degree diagonal matrices, respectively.
The {\it Laplacian polynomial} $\mathcal{L}(G,\lambda)$ of $G$ is the characteristic
polynomial of its Laplacian matrix $L(G)$, i.e.,
$$\mathcal{L}(G,\lambda)=det(\lambda I_{n}-L(G))=\sum_{k=0}^n
(-1)^{k}c_{k}(G)\lambda^{n-k}.$$
Then  $L(G)$ has
nonnegative eigenvalues
$\mu_{1}\geq\mu_{2}\geq\cdots\geq\mu_{n-1}\geq\mu_{n}=0$. From
Viette's formula,
$c_{k}=\sigma_{k}(\mu_{1},\mu_{2}\cdots\mu_{n-1})$ is a symmetric
polynomial of order $n-1$. In particular, we have
$c_{0}=1,c_{n}=0,c_{1}=2|E(G)|,c_{n-1}=n\tau(G)$, where $\tau(G)$
is the number of spanning trees of $G$ (see \cite{Merris1995}).
If $G$ is a tree, the Laplacian coefficient $c_{n-2}$ is equal to
its Wiener index, which is the sum of all distances between
unordered pairs of vertices of $G$ (see \cite{Dobrynin2001}, \cite{Yan2006}).
$$c_{n-2}(T)=W(T)=\sum_{u,v\in V} d(u,v).$$
  In general, Laplacian coefficients $c_{k}$ can be expressed in terms of subtree
structures of $G$.

\begin{theorem}\cite{Kelmans1974}\label{theorem1.1}
  Let $\mathcal {F}_{k}$
be the set of all spanning forests of $G$ with exactly $k$
components.
The Laplacian coefficient $c_{n-k}$ of a graph $G$ is expressed by
$c_{n-k}(G)=\sum_{F\in \mathcal {F}
_{k}} \gamma(F)$, where $F$ has $k$ components
$T_{i}$ with $n_{i}$ vertices, $i=1,2,\cdots,k$ and  $\gamma(F)=\prod_{i=1}^k n_{i}$.
\end{theorem}
Recently, the study on the Laplacian coefficients has attracted much
attention. Let $G$ and $H$ be two graphs of order $n$. We write $G\preceq H$ if
$c_{k}(G)\leq c_{k}(H)$ for all $0\leq k\leq n$, and  write
$G\prec H$ if $G\preceq H$ and $c_{k}(G)<c_{k}(H)$ for some $k\in
\{0,1,\cdots,n\}$.
 Mohar \cite{Mohar2007} first investigated properties of the poset (partially ordered set) of
acyclic graphs with the partial order $\preceq$ and proposed some problems. Later, Ili\'{c}\cite{Ilic2009} and   Zhang et al. \cite{Zhang2009}
investigated ordering trees  by the Laplacian
coefficients. Stevanovi\'{c} and Ili\'{c} \cite{Stevanovic2009b}
investigated and characterized the minimum and maximum elements in the poset of unicyclic graphs of order $n$ with $\preceq$.  Tan \cite{Tan2011}
proved that the poset of  unicyclic  graphs of order $n$  and fixed
matching number with $\preceq$ has only one minimal element. He and
Shan \cite{He2010} studied the properties of the poset of bicyclic graphs of order $n$.

{\it The Laplacian-like energy} \cite{Liu2008} of $G$ , $LEL$ for short, is defined as
follows:
$$LEL(G)=\sum_{k=1}^{n-1} \sqrt{\mu_{k}},$$
since it has similar
features as molecular graph energy defined by Gutman \cite{Gutman1978}.
 $LEL$ describes well the properties which have a close relation with the molecular structures and was proved to be as good as the Randi\'{c} index, better than Wiener index in some areas (see \cite{Stevanovic2009c}).  Further, Stevanovi\'{c} in \cite{Stevanovic2009} established a connection between $LEL$ and
Laplacian coefficients.

\begin{theorem}\cite{Stevanovic2009c} \label{theorem1.2}
Let $G$ and $H$ be two graphs with $n$ vertices. If $G\preceq H$, then $LEL(G)\leq LEL(H)$.
Furthermore, if $G\prec H$, then $LEL(G)<LEL(H)$.
\end{theorem}
Let
 \begin{eqnarray*}
 \mathcal{U}_{n,l}&=& \{G\ |\  G \mbox{ is a $n$-vertex unicyclic graph with fixed $l$ leaves }\},\\
 \mathcal{U}_{n,l}^g&=& \{G\ |\ G\in \mathcal{U}_{n,l}  \mbox{ with fixed girth $g$ }\}.
 \end{eqnarray*}
Let  $BST_{n,l}$ be a   balanced starlike tree of order $n$ with $l$ leaves  which is obtained by identifying one end of each of the $l$ paths of orders $\lfloor\tfrac{n-1}{l}\rfloor+1$ or $\lceil\tfrac{n-1}{l}\rceil+1$. Moreover, let   $U_{n,l}^{g,p}$  (see Fig.1) be a balanced starlike unicyclic graph of order $n$ with $l$ leaves and girth $g$, which is obtained by identifying one end of  a path $P_{p+1}$ of order $p+1$ and one vertex of a cycle of order $g$, the other end of $P_{p+1}$ and the center vertex of a balanced starlike tree $BST_{n-p-g+1, l}$, respectively.
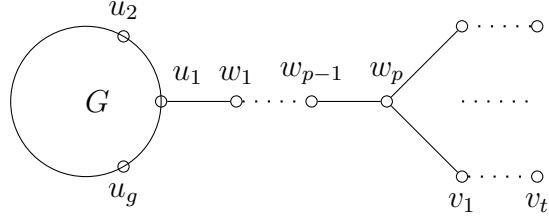
\begin{figure}
\centering
\begin{tikzpicture}
\node(v900)[label=0:$G$]at (-1.3,0){};
\node (v0)[draw,shape=circle,inner sep=1.5pt,label=87:$u_1$] at (0:0){};
\node (v1)[draw,shape=circle,inner sep=1.5pt,label=90:$u_{2}$] at(-60:-1){};
\node (v2)[draw,shape=circle,inner sep=1.5pt,label=below:$u_{g}$] at(60:-1){};
\node(v3) [draw,shape=circle,inner sep=1.5pt,label=90:$w_{1}$] at (1,0){};
\node(v4) [draw,shape=circle,inner sep=1.5pt,label=90:$w_{p-1}$] at (2,0){};
\node(v5) [draw,shape=circle,inner sep=1.5pt,label=90:$w_{p}$] at (3,0){};
\node(v6) [draw,shape=circle,inner sep=1.5pt] at (4,1){};
\node(v7) [draw,shape=circle,inner sep=1.5pt,label=-90:$v_1$]at (4,-1){};
\node(v8) [draw,shape=circle,inner sep=1.5pt] at (5,1){};
\node(v9)[draw,shape=circle,inner sep=1.5pt,label=-90:$v_t$] at (5,-1){};
\draw (-1,0) circle(1);
\draw [loosely dotted,thick]
(v8) --(v6) (4,0)--(5,0)
(v7)--(v9)
(v3)--(v4);
\draw (v0)--(v3) (v4)--(v5) (v5)--(v6) (v5)--(v7);
\end{tikzpicture}
\caption{Graph $U_{n,l}^{g,p}$} \label{fig:pepper}
\end{figure}
Ili\'{c} and Ili\'{c} in \cite{Ilic2009b} proposed the following conjecture:
\begin{conjecture}\cite{Ilic2009b}\label{conjecture}
Among all $n$-vertex unicyclic graphs, the graph $U_{n,l}^{3,0}$  has the minimum Laplacian coefficients $c_k,$ $k=0, \dots,n,$
 i.e., $U_{n,l}^{3,0}$ is the only one minimal element in the poset $(\mathcal{U}_{n,l}, \preceq)$.
\end{conjecture}
However, this conjecture is, in general, not true.  Let $G_1$ and $G_2 $ be the following two unicyclic graphs of orders 10 (see Fig. 2).
\begin{figure}
\centering
\begin{tikzpicture}[scale=1]
\tikzstyle{every node}=[draw,shape=circle,inner sep=1.5pt];
\node (v0) at (0:0){};
\node(v1) at (0:1){};
\node(v2) at(60:1){};
\node(v3) at (2,1){};
\node(v4) at (2,-1){};
\node(v5) at (3,1){};
\node(v6) [draw,shape=circle,inner sep=1.5pt,label=below:$G_1$]at (3,-1){};
\node(v7) at (4,1){};
\node(v8) at (4,-1){};
\node(v10) at (5,-1){};
\draw (v0) --(v1)
(v1) --(v2)
(v2)--(v0)
(v1)--(v4)
(v1)--(v3)
(v8)--(v10)
(v5)--(v7)
(v4)--(v6)
(v6)--(v8)
(v5)--(v3);
\end{tikzpicture}
\begin{tikzpicture}[scale=1]
\tikzstyle{every node}=[draw,shape=circle,inner sep=1.5pt];
\node (v0) at (0:0){};
\node(v1) at (0:1){};
\node(v2) at(60:1){};
\node(v3) at (2,0){};
\node(v4) at (3,1){};
\node(v5) at (3,-1){};
\node(v6) at (4,1){};
\node(v7) [draw,shape=circle,inner sep=1.5pt,label=below:$G_2$] at(4,-1){};
\node(v8) at (5,1){};
\node(v9) at (5,-1){};
\draw (v0) --(v1)
(v1) --(v2)
(v2)--(v0)
(v1)--(v3)
(v3)--(v5)
(v3)--(v4)
(v7)--(v9)
(v5)--(v7)
(v4)--(v6)
(v6)--(v8);
\end{tikzpicture}
\caption{Counter-example} \label{fig:pepper}
\end{figure}
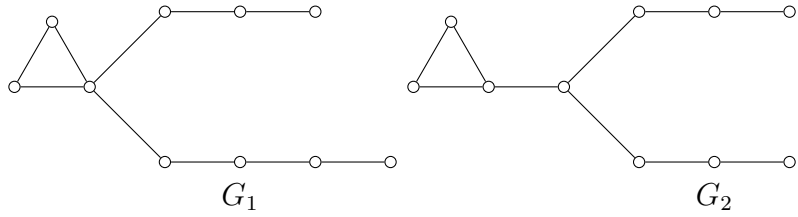
Then their Laplacian characteristic polynomials are
$$\mathcal{L}(G_1,x)=x^{10}-20x^9+167x^8-758x^7+2036x^6-3296x^5+3130x^4-1612x^3+382x^2-30x,$$
$$\mathcal{L}(G_2,x)=x^{10}-20x^9+168x^8-770x^7+2091x^6-3414x^5+3243x^4-1642x^3+373x^2-30x.$$
Hence $G_1$ does not have  minimal Laplacian coefficients $c_k,$ $k=0, \dots,10.$ So this conjecture is, in general, not true. But it is proven that $G_1$ is still a minimal element in the poset $\mathcal{U}_{10,2}$. In fact, there are many minimal elements in the poset $\mathcal{U}_{n,l}$. This paper is organized as follows: in section 2, we investigate some properties of minimal elements in the poset $\mathcal{U}_{n,l}^g$.
In sections 3 and 4, all minimal elements in four special posets of $\mathcal{U}_{n,l}^g$ are characterized, respectively. Finally, in section 5, we conclude this paper with some conjectures.

\section{The minimal elements in $(\mathcal{U}_{n,l}^g, \preceq)$}

Let $v$ be a vertex of a connected graph $G$ and let $N_{G}(v)$
denote the set of the neighbors of $v$ in $G$. Let $d_{G}(v)$
denote the degree of $v$ in $G$, if $d_{G}(v)=1$, $v$ is called
a leaf or a pendent vertex. Say that $P=vv_{1}\cdots
v_{k}$ is a pendant path of length $k$ attached at vertex $v$ if its
interval vertices $v_1\cdots v_{k-1}$ have degree two and $v_k$ is a
leaf. If $k=1$, then $v_1$ is a {\it leaf} and $vv_1$ is called a {\it pendent
edge}. A {\it branch vertex} is a vertex having degree more than two. Moreover,
let $d(u,v)$ denote the distance between vertices $u$ and $v$.
For a $n$-vertex unicyclic graph $G\in \mathcal{U}_{n,l}^g$, $G$ can be obtained from
     a cycle $C_g=u_{1}\cdots u_{g}$ of order $g$ by attaching trees $T_1\cdots T_g$ rooted at $u_{1}, \cdots, u_{g}$, respectively. So  $G$ may be written to be  $C_{T_1, \cdots, T_g}$.

\begin{lemma}\label{theorem2.1}
(\cite{Ilic2010}). Let $v$ be a vertex of a nontrivial connected graph $G$ and for
nonnegative integers $p$ and $q$, let $G(p,q)$ denote the graph
obtained from $G$ by adding two pendent paths of lengths $p$ and $q$
at $v$, respectively, $p \geq q \geq 1$. Then $c_k(G(p,q))\leq
c_k(G(p+1,q-1)), k=0,1,\cdots,n$.
\end{lemma}

 \begin{definition}\label{definition 2.2}
 Let $G$ be an arbitrary connected graph with $n$ vertices.
  If $uv$ is a non-pendent cut edge of $G$ with $d_{G}(v)\geq 3$ and $d_{G}(u)\geq 2$ such that there's at least one pendent path $P_{t+1}=vv_{1}\cdots v_{t}$  attached at $v$,
then  the  graph $G^\prime=\xi(G,uv)$ obtained from $G$  by changing all edges (except $uv, vv_1$)  incident with $v$ into new edges between $u$ and $N_G(v)\setminus \{u,v_1\}$.
 In other words, $$G^\prime=G-\{vx| x\in N_G(v)\setminus \{u,v_1\}\}+\{ux| x\in N_G(v)\setminus \{u,v_1\}\}.$$
 We say that $G^\prime$ is a $\xi$-transformation of $G$. (See Fig. 3).
 \end{definition}
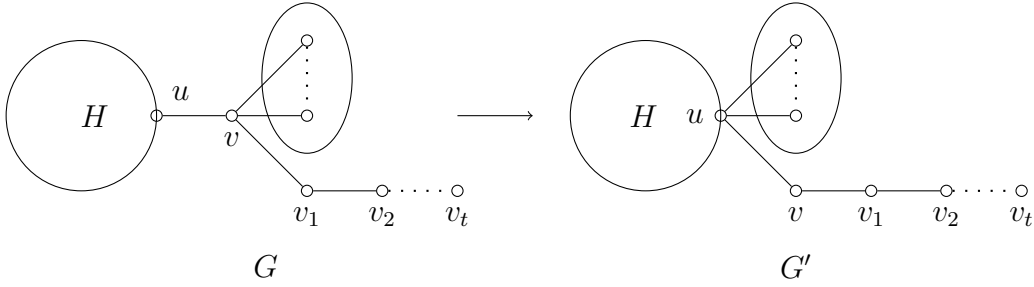
\begin{figure}
\centering
\begin{tikzpicture}
\node(v900)[label=0:$H$]at (-1.3,0){};
\node (v0)[draw,shape=circle,inner sep=1.5pt,label=45:$u$] at (0:0){};
\node[draw,shape=circle,inner sep=1.5pt,label=below:$v$](v1) at (1,0){};
\node[draw,shape=circle,inner sep=1.5pt](v2) at (2,1){};
\node[draw,shape=circle,inner sep=1.5pt](v3) at (2,0){};
\node[draw,shape=circle,inner sep=1.5pt,label=below:$v_1$](v4) at (2,-1){};
\node[draw,shape=circle,inner sep=1.5pt,label=below:$v_2$](v5) at (3,-1){};
\node[draw,shape=circle,inner sep=1.5pt,label=below:$v_t$](v6) at (4,-1){};

\draw [loosely dotted,thick] (v2) -- (v3)
(v5) --(v6);
\draw (v0)--(v1)(v1)--(v2) (v1)--(v3)(v1)--(v4)(v4)--(v5);
\draw  (2,0.5) ellipse (0.6 and 1);
\draw (-1,0) circle(1);
\node(v900)[label=0:$H$]at (6,0){};
\node (v0)[draw,shape=circle,inner sep=1.5pt,label=180:$u$] at (7.5,0){};
\node[draw,shape=circle,inner sep=1.5pt,label=below:$v$](v1) at (8.5,-1){};
\node[draw,shape=circle,inner sep=1.5pt](v2) at (8.5,1){};
\node[draw,shape=circle,inner sep=1.5pt](v3) at (8.5,0){};
\node[draw,shape=circle,inner sep=1.5pt,label=below:$v_1$](v4) at (9.5,-1){};
\node[draw,shape=circle,inner sep=1.5pt,label=below:$v_2$](v5) at (10.5,-1){};
\node[draw,shape=circle,inner sep=1.5pt,label=below:$v_t$](v6) at (11.5,-1){};

\draw [loosely dotted,thick] (v2) -- (v3)
(v5) --(v6);
\draw (v0)--(v2)(v0)--(v3) (v0)--(v1)(v1)--(v4)(v4)--(v5);
\draw  (8.5,0.5) ellipse (0.6 and 1);
\draw (6.5,0) circle(1);
\draw[->] (4,0)--(5,0);
\node(v900)[label=0:$G$]at (1,-2){};
\node(v900)[label=0:$G'$]at (8,-2){};
\end{tikzpicture}
\caption{$\xi$-transformation } \label{fig:pepper}
\end{figure}
Clearly, $\xi$-transformation preserves the number of leaves in $G$.
 \begin{lemma}\label{lemma2.3}
 Let $G$ be a connected graph of order $n$ and $G^\prime$ be obtained from $G$ by $\xi$-transformation. If  there exists a path $P_{s+1}=uu_1^p\cdots u_{s-1}^pu_s^p$ of order $s+1$ in the component of $G-uv$ and  a pendent path $P_{t+1}=vv_{1}\cdots v_{t}$  attached at $v$ with $s\ge t$, then $G^\prime\preceq G$, i.e,
 $$c_k(G)\geq c_k(G^{\prime}),  k=0, 1,\dots, n,$$
 with equality if and only if $k \in \{0,1,n-1,n\}$.
 \end{lemma}
 \begin{proof}
Clearly, if $k\in \{0,1,n\}$, $c_k(G)=c_k(G^\prime)$.
 Since $uv$ is a cut edge, then every spanning tree of $G$ and $G^\prime$
 includes edge $uv$,  which implies $\tau(G)=\tau(G^\prime)$. Hence
 $c_{n-1}(G)=c_{n-1}(G^\prime)$. Now assume that $2 \leq k \leq n-2$
and consider the coefficients $c_{n-k}(G)$. Let $\mathcal {F}^\prime$ (resp. $\mathcal {F}$)
be the set of all spanning forests of $G'$ (resp. $G$) with exactly $k$ components.
For an arbitrary spanning forest $F' \in \mathcal {F}'$, denote by $T'$  the component of $F'$ containing
$u$. Let $f: \mathcal{F}^\prime\rightarrow \mathcal{F}, F=f(F^\prime)$, where $V(F)=V(F^\prime)$ and
$$E(F)=E(F')-\{ux| x\in N_{T^\prime}(u)\cap N_{G}(v)\}+\{vx| x\in N_{T^\prime}(u)\cap N_{G}(v)\}.$$
Then $f$ is injective.  Let
$\mathcal{F}^\prime =\mathcal{F}_{(1)}^\prime \bigcup \mathcal{F}_{(2)}^\prime
$, where
$\mathcal{F}_{(1)}^\prime=\{F'\in \mathcal{F}^\prime \ |\ uv\in E(F')\} $ and
$\mathcal{F}_{(2)}^\prime=\{F'\in \mathcal{F}^\prime \ |\ uv\notin E(F')\} $.
If   $F^\prime\in \mathcal{F}_{(1)}^\prime,$ then  $F'$ and $F=f(F')$ have the same components except $T^\prime$.
Moreover, $T'$ and  $f(T^{\prime})$  have the same vertices. Hence $\gamma(F)=\gamma(F')$.  If $F^\prime\in \mathcal{F}_{(2)}^\prime,$
let $S^\prime$ be the component of $F'$ containing $v$.
Then $F^\prime$ and $F=f(F^\prime)$ have the same components except $T^\prime$ and $S^\prime$ in $F^\prime$.
Assume $T^\prime$ contains $a$ vertices in the component of $G-uv$
containing $v$, $|V(T^\prime)|-a$ vertices in the component of $G-uv$ containing $u$. Then $F$ has two components $f(T^{\prime})=T$ with $|V(T^\prime)|-a$ vertices and $f(S^{\prime})=S$ with $a+|V(S^\prime)|$  vertices  corresponding to $T^\prime$ and $S'$, respectively.  Denote by  $N$  the product of the orders of all components of $F^\prime$ except $T^\prime$ and $S^\prime$. Then
\begin{align*}
\gamma(f(F^{'}))-\gamma(F^{'})
=&[(|V(T^\prime)|-a)(a+|V(S^\prime)|)-|V(T^\prime)|\cdot|V(S^\prime)|]N\\
=&(|V(T^\prime)|-a-|V(S^\prime)|)\cdot a\cdot N.
\end{align*}
Further let $\mathcal{F}_{(2)}^\prime=
\mathcal{F}_{20}^\prime\cup\mathcal{F}_{21}^\prime\cup\mathcal{F}_{22}^\prime$,
where
\begin{eqnarray*}
\mathcal{F}_{20}^\prime&=&\{F^\prime\in \mathcal{F}_{(2)}^\prime\ |    \ |V(T^\prime)|-a=|V(S^\prime)| \mbox{ or } a=0\},\\
 \mathcal{F}_{21}^\prime&=& \{F^\prime\in \mathcal{F}_{(2)}^\prime\ | \  |V(T^\prime)|-a<|V(S^\prime)|, a>0\},\\
 \mathcal{F}_{22}^\prime&=&\{F^\prime\in \mathcal{F}_{(2)}^\prime\ |\   |V(T^\prime)|-a>|V(S^\prime)|, a>0\}.
 \end{eqnarray*}
 Hence it follows  that $$\forall F^\prime\in \mathcal{F}_{20}^\prime,  \gamma(f(F^{'}))-\gamma(F^{'})=0,$$
$$\forall F^\prime\in \mathcal{F}_{21}^\prime,  \gamma(f(F^{'}))-\gamma(F^{'})<0,$$
$$\forall F^\prime\in \mathcal{F}_{22}^\prime,  \gamma(f(F^{'}))-\gamma(F^{'})>0.$$
For every spanning forest $F_1^\prime\in \mathcal{F}_{21}^\prime$,
let $T^\prime$ and  $S^\prime$ be two components  of $F_1^\prime$  containing $u$ and $ v$, respectively.  Assume that
$ u,  u_1^p, \cdots, u_{r-1}^p\in V(T')$ and $u_r^p\notin V(T')$.
Moreover, let  $R^\prime$ be  a component of $F_1^\prime$ containing $u_r^p$ with  $b$ vertices.
Thus let $T^{\prime\prime}$ be a tree obtained from $T^\prime$ and $R^\prime$ by joining
$u_{r-1}^p$ and $u_r^p$ with the edge $u_{r-1}^pu_r^p$,   $S^{\prime\prime}$ be
the path $vv_1\cdots v_{|V(T^\prime)|-a-1}$ and $R^{\prime\prime}$ be
the path $v_{|V(T^\prime)|-a}\cdots v_{|V(S^\prime)|-1}$.
Then $F_2^\prime=(F_1^\prime-\{T^\prime, S^\prime, R^\prime\})\cup \{T^{\prime\prime},
S^{\prime\prime}, R^{\prime\prime}\}$  is a spanning forest
of $G^\prime$ with exactly $k$ components  and
$F_2^\prime\in \mathcal{F}_{22}^\prime$.
Hence there exists an injective (not bijective) map from $\mathcal{F}_{21}^\prime$ to $ \mathcal{F}_{22}^\prime$, i.e.,
$$\varphi: \mathcal{F}_{21}^\prime\rightarrow \mathcal{F}_{22}^\prime: F_1^{\prime}\rightarrow F_2^{\prime}=\varphi(F_1^\prime),$$
where $F_2^{\prime}=\varphi(F_1^\prime)=(F_1^\prime-\{T^\prime, S^\prime, R^\prime\})\cup \{T^{\prime\prime}, S^{\prime\prime}, R^{\prime\prime}\}$.
Note that  $|V(T^{\prime\prime})|=|V(T^\prime)|+b,$  $|V(S^{\prime\prime})|=|V(T^\prime)|-a,$ and $ |V(R^{\prime\prime})|= |V(S^\prime)|-|V(T^\prime)|+a$. It is easy to see that for $F^\prime\in\mathcal{F}_{21}^\prime$,
\begin{eqnarray*}
&&\gamma(f(\varphi(F^\prime)))-\gamma(\varphi(F^\prime))\\
&=& [(|V(T^\prime)|+b-a)\cdot (|V(T^\prime)|-a+a)-(|V(T^\prime)|+b)\cdot (|V(T^\prime)|-a)]\\
&& \cdot (|V(S^\prime)|-|V(T^\prime)|+a)\cdot \frac{N}{b}
\\
&=& -(|V(T^\prime)|-a-|V(S^\prime)|)\cdot a N
\\
&=&-(\gamma(f(F^\prime))-\gamma(F^\prime))
\end{eqnarray*}
Therefore,
\begin{align*}
&\sum_{F^{\prime}\in \mathcal {F}_{21}^{\prime} \cup  \mathcal {F}_{22}^\prime} [ \gamma(f(F^{\prime}))-\gamma(F^{\prime})]\\
=&\sum_{F^{\prime}\in \mathcal {F}_{21}^{\prime}} [\gamma(f(F^{\prime}))-\gamma(F^{\prime})+\gamma(f(\varphi(F^\prime)))-\gamma(\varphi(F^\prime))]+\sum_{F^\prime\in \mathcal {F}_{22}^{\prime}\setminus
\varphi(\mathcal {F}_{21}^{\prime})} [\gamma(f(F^{\prime}))-\gamma(F^{\prime})]\\
=&\sum_{F^\prime\in \mathcal {F}_{22}^{\prime}\setminus
\varphi(\mathcal {F}_{21}^{\prime})} [\gamma(f(F^{\prime}))-\gamma(F^{\prime})]>0.
\end{align*}
It follows from  Theorem~\ref{theorem1.1} that
$$c_{n-k}(G^{\prime})=\sum_{F^{\prime}\in \mathcal {F}^{\prime}}\gamma(F^{\prime}) <
\sum_{F^{\prime}\in \mathcal {F}^{\prime}}\gamma(f(F^{\prime}))\leq \sum_{F\in \mathcal {F}}\gamma(F)=c_{n-k}(G).$$
So the assertion holds.\end{proof}


Now we are ready to present the main result in this section.
\begin{theorem}\label{theorem2.5}
 If $C_{T_1,\cdots,T_g}\in \mathcal{U}_{n,l}^g$ is obtained from the cycle $C_g=u_1\cdots u_g$ by attaching $g$ trees $T_1, \cdots, T_g$
at the roots $u_1, \cdots, u_g$, respectively, where  $|V(T_i)|=n_i$ and the number of leaves in $T_i$ is $l_i$,  then
$$C_{B_1,\cdots,B_g}\preceq C_{T_1,\cdots,T_g},$$
where  $B_i$ is a tree with root $u_i$ obtained by identifying one end $u_{i,p_i}$ of the path $P_{p_i}: u_iu_{i,2}\cdots u_{i,p_i}$ and the center of $BST_{n_i-p_i+1, l_i}$  for $i=1, \cdots, g$ and $p_1, \cdots, p_g$ are equal to 1 except at most one $p_j$. Moreover,
the equality holds  if and only if $C_{B_1,\cdots,B_g}\cong C_{T_1,\cdots,T_g}$.
\end{theorem}

\begin{proof}
Denote $G= C_{T_1,\cdots,T_g}$. We prove this result by two steps.

{\bf Step 1:} Let $P$ be  the longest path among all paths which start at $u_i$ in $T_i$ for $i=1, \cdots, g$. Without loss of generality, assume
$P$ belongs to $T_1$. Let $v$ be a farthest branch vertex from vertex $u_i$ in $T_i$ for $i=2, \cdots, g$. Thus there exists a cut edge $uv$ in $T_i$ such that $d(u_i,u)+1=d(u_i,v)$. By  $\xi$-transformation on $uv$ and Lemma~\ref{lemma2.3}, we obtain
$G_1=\xi(G,uv)=C_{T_1,T_2, \cdots, T_i', \cdots, T_g}\preceq G$ such that the number of branch vertices in $T_i$ is non-increasing. Hence after a series of $\xi$-transformations
on the cut edge $xy$ with $y$ being a branch vertex and $d(u_i, x)+1=d(u_i, y)$ in $T_i'$,
$G_2=C_{T_1, T_2, \cdots, T_i^{\prime,\prime}, \cdots, T_g}$ is obtained,
where $T_i^{\prime,\prime}$ is a starlike tree of order $n_i$ and leaves $l_i$ with root $u_i$.
Hence  for $i=2, \cdots, g$, after performing  a series of $\xi$-transformations and repeatedly using Lemma~\ref{theorem2.1}, there exists an unicyclic graph $G_3=C_{T_1, B_2, \cdots, B_g} $ such that  $G_3\preceq G_1$  where $B_i$ is $ BST_{n_i, l_i}$ with a center vertex $u_i$ for $i=2, \cdots, g$.

{\bf Step 2:}  If $T_1$ is a path or a starlike tree of order $n_1$, then the assertion holds by using Lemma~\ref{theorem2.1}.
Assume that $T_1$ has at least one branch vertex except the root $u_1$.
Let $v$ be the branch vertex which is nearest to some pendent vertices in $T_1$ (maybe $v$ is not unique).
Then there exists a cut edge $uv$, $d(u_1,u)+1=d(u_1,v)$.
If there exists $uv$ as defined above which satisfies the conditions of Lemma~\ref{lemma2.3},
then $G_4=\xi(G_3,uv)\preceq G_3$.
Moreover, the number of branch vertices in $G_4$ is no more than that in $G_3$.
After performing a series of this type of $\xi$-transformations and repeatedly using Lemma~\ref{theorem2.1},
there exists an unicyclic graph $G_5=C_{T_1^\prime, B_2, \cdots, B_g} $ such that $G_5\preceq G_4$, where $T_1^\prime$ is a tree rooted at
$u_1$ obtained by attaching $l_1-1$ pendent paths at some vertices of the longest path $P^\prime$ of $T_1^\prime$.
If all pendent paths are attached at the only one vertex $u_1$ or $u_{1,x}$ of $P^\prime$, then the result holds.

Otherwise, let $v^\prime$ be the branch vertex which is nearest to $u_1$ in $T_1^\prime$ (when $d_{G_5}(u_1)> 3$, $v^\prime=u_1$).
Then there exists a cut edge $u^\prime v^\prime$ which satisfies $d(u_1,v^\prime)+1=d(u_1,u^\prime)$.
By $\xi$-transformation on $u^\prime v^\prime$ and Lemma~\ref{lemma2.3},
we obtain $G_6=\xi(G_5,u^\prime v^\prime)\preceq G_5$.
Further, the number of branch vertices in $G_6$ is no more than that in $G_5$.
Hence by performing a series of this type of $\xi$-transformations,
$G_7=C_{B_1, \cdots, B_g}$ is obtained, where $B_1$ has exactly one branch vertex $u_1$ or $B_1$ has exactly one branch vertex $u\neq u_1$ with $d_{G_7}(u_1)=3$.  If $u=u_1$, then by Lemma~\ref{theorem2.1}, the assertion holds. If $u\neq u_1$ and $d_{G_7}(u_1)=3$, then applying Lemma~\ref{theorem2.1} to all pendent paths in $G_7$ yields the desired result.
\end{proof}

\begin{corollary}\label{cor2.6}
  Let $C_{T_1,\cdots,T_g}\in \mathcal{U}_{n,l}^g$ be obtained from the cycle $C_g=u_1\cdots u_g$ by attaching $g$ trees $T_1, \cdots, T_g$
  at the roots $u_1, \cdots, u_g$, respectively. If  $d(u_1)> 3$ and $|V(T_i)|=1$  for $i=2, \cdots, g$, then
$$C_{B_1,\cdots,B_g}\preceq C_{T_1,\cdots,T_g},$$
where  $B_1$ is  $BST_{n-g+1, l}$  with a center vertex $u_1$ and $|V(B_i)|=1$ for $i=2, \cdots, g.$  Moreover,
the equality holds  if and only if $C_{B_1,\cdots,B_g}\cong C_{T_1,\cdots,T_g}$.
\end{corollary}
\begin{proof}
It is obvious that the assertion  follows from the proof of Theorem~\ref{theorem2.5}.\end{proof}

\section{The minimal elements in two  subsets of $\mathcal{U}_{n,l}^g$}
In this section, we characterize all extremal graphs which have minimal Laplacian coefficients in the following two special subsets of $\mathcal{U}_{n,l}^g$. Denote
\begin{eqnarray*}
\mathcal{U}_{n,l}^{g,1}&=&\{C_{T_1,\cdots,T_g}| \  |V(T_i)|=1\  \mbox{ for}\  i=2, \cdots,  g\},\\
\mathcal{U}_{n,l}^{g,2}&=&\{C_{T_1,\cdots,T_g}| \  |V(T_1)|>1, |V(T_i)|>1, |V(T_j)|=1 \ \mbox{for }   j\neq 1, i \}.
\end{eqnarray*}
Clearly, $U_{n, l}^{g, p}$ is in  $\mathcal{U}_{n,l}^{g,1}$. For convenience, denote $U_{n,l}^{g,0}=U^0, U_{n,l}^{g,1}=U^1, \cdots, U_{n,l}^{g,p}=U^p$.

\begin{lemma}\label{lemma3.1}
For $1\le p\le \lfloor\frac{n-g-gl+l}{l+1}\rfloor$,  $U^p$ and $U^{p-1}$ are incomparable
 in the poset  $(\mathcal{U}_{n,l}^{g}, \preceq)$. (See Fig.1).
\end{lemma}
\begin{proof}
   We first show that $c_{n-2}(U^p)<c_{n-2}(U^{p-1})$.
It is obvious that $U^{p-1}$ can be regarded as $U^{p-1}=\xi(U^p, w_{p-1}w_p)$.
For convenience, we denote $U^p$ and $U^{p-1}$ by $G$ and $G^\prime$, respectively.
Let $\mathcal{F}_2$ (resp. $\mathcal{F}_2^{\prime}$) be the set of all spanning forests of $G$ (resp. $G^\prime$)
with exactly $2$ components. For an arbitrary spanning forest $F\in \mathcal{F}_{2}$
(resp. $F^\prime\in \mathcal{F}_{2}^\prime$), $F$ (resp. $F^\prime$) can be obtained by
deleting two edges $\{e_1, e_2\}$ in $E(G)$ (resp. $E(G^\prime)$) with $e_1$ belonging to the cycle in $G$ (resp. $G^\prime$).
If $e_2\neq w_{p-1}w_p$, then $F$ and $F^\prime$ have the same components, which implies $\gamma(F)=\gamma(F^\prime)$.
If $e_2= w_{p-1}w_p$, then
$$\gamma(F)-\gamma(F^\prime)=(g+p-1)(n-g-p+1)-(\lfloor\frac{n-g-p}{l}\rfloor+1)(n-\lfloor\frac{n-g-p}{l}\rfloor-1)<0,$$
Therefore, $$c_{n-2}(G)-c_{n-2}(G^\prime)=\sum _{F\in \mathcal{F}_2}\gamma (F)
-\sum _{F^\prime\in \mathcal{F}_2^{\prime}}\gamma (F^\prime)<0.$$
We  next show $c_{m}(G)> c_{m}(G^\prime)$ for $2p\leq m\leq 2(p+g)-3$.
 Clearly $P_{p+g-1}=w_{p-1}w_{p-2}\cdots w_1u_1u_2\cdots u_g$ (denote $w_i=u_{-i+1}, 1\leq i\leq p-1$ for convenience)
is the longest path of order $p+g-1$ in the component of $G-w_{p-1}w_p$.
Let $\mathcal{F}^{\prime}$ (resp. $\mathcal{F}$) be the set of all spanning forests of $G^\prime$ (resp. $G$)
with exactly $n-m$ components, in other words, $\mathcal{F}^{\prime}$ (resp. $\mathcal{F}$) is the set of
all spanning forests of $G^\prime$ (resp. $G$) with exactly $m$ edges.
For an arbitrary spanning forest $F^\prime\in \mathcal{F}^\prime$, denote by $T'$  the component of $F'$ containing
$w_{p-1}$. Let $f: \mathcal{F}^\prime\rightarrow \mathcal{F}, F=f(F^\prime)$, where $V(F)=V(F^\prime)$ and
$$E(F)=E(F')-\{w_{p-1}x| x\in N_{T^\prime}(w_{p-1})\cap N_{G}(w_p)\}+\{w_px| x\in N_{T^\prime}(w_{p-1})\cap N_{G}(w_p)\}.$$
Then $f$ is injective.  Let
$\mathcal{F}^\prime =\mathcal{F}_{(1)}^\prime \bigcup \mathcal{F}_{(2)}^\prime$, where
$\mathcal{F}_{(1)}^\prime=\{F'\in \mathcal{F}^\prime \ |\ w_{p-1}w_p\in E(F')\} $ and
$\mathcal{F}_{(2)}^\prime=\{F'\in \mathcal{F}^\prime \ |\ w_{p-1}w_p\not\in E(F')\} $.
If   $F^\prime\in \mathcal{F}_{(1)}^\prime,$ then  $F'$ and $F=f(F')$ have the same components except $T^\prime$.
Moreover, $T'$ and  $f(T^{\prime})$  have the same vertices. Hence $\gamma(F)=\gamma(F')$.  If $F^\prime\in \mathcal{F}_{(2)}^\prime,$
let $S^\prime$ be the component of $F'$ containing $w_p$.
Then $F^\prime$ and $F=f(F^\prime)$ have the same components except $T^\prime$ and $S^\prime$ in $F^\prime$.
Assume $T^\prime$ contains $a$ vertices in the component of $G-w_{p-1}w_p$
containing $w_p$, $|V(T^\prime)|-a$ vertices in the component of $G-w_{p-1}w_p$ containing $w_{p-1}$. Then $F$ has two components $f(T^{\prime})=T$ with $|V(T^\prime)|-a$ vertices and $f(S^{\prime})=S$ with $a+|V(S^\prime)|$  vertices  corresponding to $T^\prime$ and $S'$, respectively.  Denote by  $N$  the product of the orders of all components of $F^\prime$ except $T^\prime$ and $S^\prime$. Then
\begin{align*}
\gamma(f(F^{'}))-\gamma(F^{'})
=&[(|V(T^\prime)|-a)(a+|V(S^\prime)|)-|V(T^\prime)|\cdot|V(S^\prime)|]N\\
=&(|V(T^\prime)|-a-|V(S^\prime)|)\cdot a\cdot N.
\end{align*}
Further let $\mathcal{F}_{(2)}^\prime=
\mathcal{F}_{20}^\prime\cup\mathcal{F}_{21}^\prime\cup\mathcal{F}_{22}^\prime$,
where
\begin{eqnarray*}
\mathcal{F}_{20}^\prime&=&\{F^\prime\in \mathcal{F}_{(2)}^\prime\ |    \ |V(T^\prime)|-a=|V(S^\prime)| \mbox{ or } a=0\},\\
 \mathcal{F}_{21}^\prime&=& \{F^\prime\in \mathcal{F}_{(2)}^\prime\ | \  |V(T^\prime)|-a<|V(S^\prime)|, a>0\},\\
 \mathcal{F}_{22}^\prime&=&\{F^\prime\in \mathcal{F}_{(2)}^\prime\ |\   |V(T^\prime)|-a>|V(S^\prime)|, a>0\}.
 \end{eqnarray*}
 Hence it follows  that $$\forall F^\prime\in \mathcal{F}_{20}^\prime,  \gamma(f(F^{'}))-\gamma(F^{'})=0,$$
$$\forall F^\prime\in \mathcal{F}_{21}^\prime,  \gamma(f(F^{'}))-\gamma(F^{'})<0,$$
$$\forall F^\prime\in \mathcal{F}_{22}^\prime,  \gamma(f(F^{'}))-\gamma(F^{'})>0.$$
For every spanning forest $F_1^\prime\in \mathcal{F}_{21}^\prime$,
let $T^\prime$ and  $S^\prime$ be two components  of $F_1^\prime$  containing $w_{p-1}$ and $w_p$, respectively.
Then  $T^\prime$ does not contain all the vertices of the path $P_{p+g-1}$.
 In fact, if all the vertices of $P_{p+g-1}$ belong to $T^\prime$,
then by the definition of $\mathcal{F}_{21}^\prime$, we have
$|V(T^\prime)|=p+g-1+a$, $|E(T^\prime)|\geq p+g-1$, and $|V(S^\prime)|>|V(T^\prime)|-a=p+g-1$,
$|E(S^\prime)|>p+g-2$, which implies $m\geq |E(T^\prime)|+|E(S^\prime)|> p+g-1+p+g-2=2(p+g)-3$. It is a contradiction to  $m\leq 2(p+g)-3$.
Therefore, assume that
$u_{-(p-1)+1}, u_{-(p-2)+1}, \cdots, u_{r-1}\in V(T')$ and $u_r\notin V(T')$.
Moreover, let  $R^\prime$ be  a component of $F_1^\prime$ containing $u_r$ with  $b$ vertices.
Thus let $T^{\prime\prime}$ be a tree obtained from $T^\prime$ and $R^\prime$ by joining
$u_{r-1}$ and $u_r$ with edge $u_{r-1}u_r$,   $S^{\prime\prime}$ be
the path $w_pv_1\cdots v_{|V(T^\prime)|-a-1}$ and $R^{\prime\prime}$ be
the path $v_{|V(T^\prime)|-a}\cdots v_{|V(S^\prime)|-1}$.
Then $F_2^\prime=(F_1^\prime-\{T^\prime, S^\prime, R^\prime\})\cup \{T^{\prime\prime},
S^{\prime\prime}, R^{\prime\prime}\}$  is a spanning forest
of $G^\prime$ with exactly $m$ edges and
$F_2^\prime\in \mathcal{F}_{22}^\prime$.
Hence there exists an injective map from $\mathcal{F}_{21}^\prime$ to $ \mathcal{F}_{22}^\prime$, i.e.,
$$\varphi: \mathcal{F}_{21}^\prime\rightarrow \mathcal{F}_{22}^\prime: F_1^{\prime}\rightarrow F_2^{\prime}=\varphi(F_1^\prime),$$
where $F_2^{\prime}=\varphi(F_1^\prime)=(F_1^\prime-\{T^\prime, S^\prime, R^\prime\})\cup \{T^{\prime\prime}, S^{\prime\prime}, R^{\prime\prime}\}$.
Note that $|V(T^{\prime\prime})|=|V(T^\prime)|+b,$  $|V(S^{\prime\prime})|=|V(T^\prime)|-a,$ and $ |V(R^{\prime\prime})|= |V(S^\prime)|-|V(T^\prime)|+a$. It is easy to see that for $F^\prime\in\mathcal{F}_{21}^\prime$,
\begin{eqnarray*}
&&\gamma(f(\varphi(F^\prime)))-\gamma(\varphi(F^\prime))\\
&=& [(|V(T^\prime)|+b-a)\cdot (|V(T^\prime)|-a+a)-(|V(T^\prime)|+b)\cdot (|V(T^\prime)|-a)]\\
&& \cdot (|V(S^\prime)|-|V(T^\prime)|+a)\cdot \frac{N}{b}
\\
&=& -(|V(T^\prime)|-a-|V(S^\prime)|)\cdot a N
\\
&=&-(\gamma(f(F^\prime))-\gamma(F^\prime))
\end{eqnarray*}
Therefore,
\begin{align*}
&\sum_{F^{\prime}\in \mathcal {F}_{21}^{\prime} \cup  \mathcal {F}_{22}^\prime} [ \gamma(f(F^{\prime}))-\gamma(F^{\prime})]\\
=&\sum_{F^{\prime}\in \mathcal {F}_{21}^{\prime}} [\gamma(f(F^{\prime}))-\gamma(F^{\prime})+\gamma(f(\varphi(F^\prime)))-\gamma(\varphi(F^\prime))]+\sum_{F^\prime\in \mathcal {F}_{22}^{\prime}\setminus
\varphi(\mathcal {F}_{21}^{\prime})} [\gamma(f(F^{\prime}))-\gamma(F^{\prime})]\\
=&\sum_{F^\prime\in \mathcal {F}_{22}^{\prime}\setminus
\varphi(\mathcal {F}_{21}^{\prime})} [\gamma(f(F^{\prime}))-\gamma(F^{\prime})]\geq 0.
\end{align*}
It follows from  Theorem~\ref{theorem1.1} that for $m\leq 2(p+g)-3$,
$$c_{m}(G^{\prime})=\sum_{F^{\prime}\in \mathcal {F}^{\prime}}\gamma(F^{\prime}) \leq
\sum_{F^{\prime}\in \mathcal {F}^{\prime}}\gamma(f(F^{\prime}))\leq \sum_{F\in \mathcal {F}}\gamma(F)=c_{m}(G).$$
Further we show that the above inequality is strict for $m\geq 2p$. In other words,
we show that $\varphi$ is not a bijective map for $2p\leq m\leq 2(p+g)-3$.
Hence it is sufficient to find a spanning forest $\bar{F}_2^{\prime}\in \mathcal {F}_{22}^{\prime}$,
such that $\bar{F}_2^{\prime}\not\in \varphi(\mathcal {F}_{21}^{\prime})$. Let
$\bar{T}^{\prime\prime}, \bar{S}^{\prime\prime}\in \bar{F}_2^{\prime}$, where $\bar{T}^{\prime\prime}$ is the path:
$x_1w_{p-1}w_{p-2}\cdots w_1u_1u_g (x_1\in N_G(w_p)\setminus \{w_{p-1}, v_1)\})$, $\bar{S}^{\prime\prime}$ is the path:
$w_pv_1\cdots v_{p-1}$. The rest $m-2p$ edges are chosen from the edge set $E(G)\setminus
(\{w_{p-1}w_p,u_1u_2, u_gu_{g-1}, v_{p-1}v_{p}\}\cup E(\bar{T}^{\prime\prime})\cup E(\bar{S}^{\prime\prime}))$
of order $n-2p-4$, it is obvious that $n-2p-4\geq m-2p$. Suppose that
there exists $\bar{F}_1^{\prime}\in \mathcal {F}_{21}^{\prime}$,
and $\varphi(\bar{F}_1^{\prime})=\bar{F}_2^{\prime}$. Let $\bar{T}^{\prime}$ be the component of $\bar{F}_1^{\prime}$ which
contains $w_{p-1}$. By the definition of $\varphi$, $|V(\bar{S}^{\prime\prime})|=|V(\bar{T}^{\prime})|-a,$
so $u_2$ is the first vertex on $P_{p+g-1}$ which does not belong to $\bar{T}^{\prime}$, then $u_1u_2\in \varphi(\bar{F}_1^{\prime})$.
Since $u_1u_2\not\in \bar{F}_2^{\prime}$, then $\varphi(\bar{F}_1^{\prime})\neq \bar{F}_2^{\prime}$, a contradiction.
 This completes the proof of Lemma~\ref{lemma3.1}.
\end{proof}

\begin{lemma}\label{lemma3.2}
For $1\le p\neq q\le \lfloor\frac{n-g-gl+l}{l+1}\rfloor$,  $U^p$ and $U^{q}$ are incomparable  in the poset  $(\mathcal{U}_{n,l}^{g}, \preceq)$.
\end{lemma}

\begin{proof}
Without loss of generality, assume $p-q= h \geq1$.
By using Lemma~\ref{lemma3.1} repeatedly, we have
$$c_{n-2}(U^{p})< c_{n-2}(U^{p-1})<\cdots<c_{n-2}(U^{q}),$$
$$ c_{m}(U^{p})\geq c_{m}(U^{p-1})\geq\cdots c_{m}(U^{q+1})>c_{m}(U^{q}), \mbox{ for } 2(q+1)\leq m\leq 2(q+g+1)-3.$$
Thus  the assertion holds.
\end{proof}

Now we are ready to characterize all minimal elements in $\mathcal{U}_{n,l}^{g,1}$.
\begin{theorem}\label{theorem3.3}
 There are exactly $p+1$ minimal elements  $U^0, \cdots, U^p$ in  $\mathcal{U}_{n,l}^{g,1}$, where $p=\lfloor\frac{n-g-gl+l}{l+1}\rfloor$.
  \end{theorem}
\begin{proof}
For any $G=C_{T_1, T_2, \cdots,T_g}\in \mathcal{U}_{n,l}^{g,1}$ with $|V(T_i)|=1$ for $i=2, \cdots, g$,  by Theorem~\ref{theorem2.5}, there exists a graph $G\succeq G_1=
C_{B_1, B_2, \cdots, B_g}\in \mathcal{U}_{n,l}^{g,1}$, where
$B_1$ is a tree with root $u_1$ obtained by identifying one end $u_{1,p_1}$ of the path $P_{p_1}: u_1u_{1,2}\cdots u_{1,p_1}$
and the center of $BST_{n_1-p_1+1, l_1}$ and $|V(B_i)|=1$ for $i=2, \cdots, g$.
If $p_1> \lfloor\frac{n-g-gl+l}{l+1}\rfloor$, then by Lemma~\ref{lemma2.3}, there exists an cut edge $uv$ such that $G_1\succeq G_2=\xi(G_1, uv)=C_{ B_1', B_2, \cdots, B_g}\in \mathcal{U}_{n,l}^{g,1}$, where $B_1'$ is a tree with root $u_1$ obtained by identifying one end $u_{1,p_1-1}$
of the path $P_{p_1-1}: u_1u_{1,2}\cdots u_{1,p_1-1}$ and the center of $BST_{n_1-p_1+2, l_1}$.
After a series of $\xi$-transformations, there exists a $U^p$ such that $G\succeq U^p$.
If  $p_1\leq\lfloor\frac{n-g-gl+l}{l+1}\rfloor$, then $G\succeq G_1=U^{p_1}.$  On the other hand, by Lemma~\ref{lemma3.2},  $U^0, \cdots, U^p$  are incomparable  in  the poset $(\mathcal{U}_{n,l}^{g,1}, \preceq)$. Hence $U^0, \cdots, U^p$ are exactly all minimal elements in the poset $(\mathcal{U}_{n,l}^{g,1}, \preceq)$.
\end{proof}

\begin{theorem}\label{theorem3.4}
For any $G=C_{T_1,\cdots,T_g}\in  \mathcal{U}_{n,l}^{g,2}$,
 $U^{0}\prec G$
\end{theorem}
\begin{proof} By Theorem~\ref{theorem2.5}, we may assume that
$G=C_{B_1, \cdots, B_i, \cdots, B_g}\in \mathcal{U}_{n,l}^{g,2}$, where
$B_1$  is a tree with root $u_1$ obtained by identifying one end $u_{1,p_1}$ of the path $P_{p_1}: u_1u_{1,2}\cdots u_{1,p_1}$
and the center vertex of $BST_{n_1-p_1+1, l_1}$, $B_i$ is $BST_{n_i, l_i}$ with a center vertex $u_i$,
  and $|V(B_j)|=1 $ for $j\neq 1, i.$
  Let
$$G^{\prime}=G-\{u_ix| x\in N_G(u_i)\setminus \{u_{i+1}, u_{i-1}\}\}+ \{u_1x| x\in N_G(u_i)\setminus \{u_{i+1}, u_{i-1}\}\}.$$
We will prove $c_k(G)\geq c_k(G^\prime)$  with at least one strict inequality.
 Clearly, when $k\in \{0,1,n-1,n\}$, $c_k(G)=c_k(G^{\prime})$.
For $2\leq k\leq n-2$, let $\mathcal {F}$ and $\mathcal{F}^{\prime}$
be the spanning forests of $G$ and $G^{\prime}$ with exactly $k$ components, respectively. For
an arbitrary spanning forest $F^{\prime}\in \mathcal{F}^{\prime}$ with $T^{\prime}$ being the component of $F^{\prime}$ containing $u_1$,
let $f: \mathcal{F}^\prime\rightarrow \mathcal{F}, F^\prime\rightarrow F=f(F^\prime)$, where $V(F)=V(F^\prime)$, and
$$E(F)=E(F^{\prime})-\{u_1x|x\in N_{G}(u_i)\cap N_{T^{\prime}}(u_1)\setminus V(C_g)\}+\{u_ix|x\in N_{G}(u_i)\cap N_{T^{\prime}}(u_1)\setminus V(C_g)\}.$$ Then $f$ is injective from $\mathcal{F}^\prime$ to $\mathcal{F}$. Denote
by $N$ the product of the orders of all components containing no $u_1, u_i$. If $u_i \in T^{\prime}$, then $F^\prime$ and $F$ have the same components except for $T^\prime$. Moreover, $F$ has
a  component containing $u_1$ which corresponds to $T^\prime$ in $F^\prime$.
Clearly, the two components have the same orders. Hence $\gamma(F)=\gamma(F^{\prime})$.
  If  $u_i\not\in T^{\prime}$, assume $u_i$ is in a component $S^\prime$ of $F^\prime$. Moreover,  there are $b\geq 0$ vertices in the connected
component containing $u_2$ in $T^\prime-u_1u_2$, and $d\geq 0$ vertices in the connected
component containing $u_g$ in $T^\prime-u_1u_g$, $e_1\geq 1$ vertices (including $u_1$) in the vertex set $V(T_1)$ and $e_2\geq 0$ vertices in the vertex set $V(T_i)\setminus \{u_i\}$.
Furthermore, the tree $S^{\prime}$ contains $a\geq 1$ vertices in the connected
component containing $u_{i}$ in $S^\prime-u_iu_{i+1}$ and $c\geq 0$ vertices in the connected
component containing $u_{i+1}$ in $S^\prime-u_iu_{i+1}$.
Then $F^\prime$ and $F$ have the same components except for $T^\prime$ and $S^\prime$. Moreover, $F$ have
two trees $T$ containing $u_1$ and $S$ containing $u_i$ which correspond to $T^\prime$ and $S^\prime$ in $F^\prime$, respectively. Hence
$$\gamma(f(F^\prime))-\gamma(F^{\prime})=[(a+c+e_2)(b+d+e_1)-(a+c)(b+d+e_1+e_2)]N= e_2(b+d+e_1-a-c)N.$$
Consider the subset $\bar{\mathcal{F}}^{\prime}$ of those spanning forests $F^{\prime}$ with $k$ components which coincide on $G^{\prime}\backslash (T^{\prime}\cup S^{\prime})$ with fixed values $e_1, e_2$. Since the two parts of $F^{\prime}$ on the cycle between $T^{\prime}$ and $S^{\prime}$ may be translated, let $a+b=M_1, c+d=M_2$ be fixed.
 Then
\begin{eqnarray*}
&&\sum_{F^{\prime}\in \bar{\mathcal{F}}^{\prime}, a+b=M_1, c+d=M_2} (\gamma(f(F^\prime))-\gamma(F^{\prime}))\\
&=&\sum_{ a+b=M_1, c+d=M_2} e_2(b+d+e_1-a-c)N\\
&=&e_2N\sum_{c=0}^{M_2} \sum_{b=0}^{M_1-1} (2b+M_2+e_1-2c-M_1)\\
&=&e_2 N M_1 \sum_{c=0}^{M_2} (e_1+M_2-2c-1)\\
&=&e_2 N M_1(e_1-1)(M_2+1)\\
&\geq& 0.
\end{eqnarray*}
Hence
$$\sum_{F^{\prime}\in \mathcal{F}^{\prime}}( \gamma(f(F^\prime))-\gamma(F^{\prime}))
=\sum_{e_1}\sum_{e_2}\sum_{M_1}\sum_{M_2}\sum_{F^{\prime}\in \bar{\mathcal{F}}^{\prime},a+b=M_1, c+d=M_2} (\gamma(f(F^\prime))-\gamma(F^{\prime}))\ge 0.$$
Since $|V(T_1)|>1, |V(T_i)|>1$, there exists one forest $F^{\prime}$ such that $e_1 >1$ and $e_2>0$. Therefore $$c_{n-k}(G^{\prime})=\sum_{F^{\prime}\in\mathcal {F}^{\prime}}\gamma(F^\prime)< \sum_{F^\prime\in \mathcal {F}^\prime}\gamma(f(F^\prime))\leq \sum_{F\in \mathcal {F}}\gamma(F)=c_{n-k}(G), 2\leq k\leq n-2.$$
Hence by Corollary~\ref{cor2.6}, we have $U^0\preceq G^\prime$ with equality if and only if $G^\prime\cong U^0$.
Therefore, $U^0\preceq G^\prime\preceq G$, the assertion holds.
\end{proof}

It  follows from Theorems~\ref{theorem1.2}, \ref{theorem3.3} and \ref{theorem3.4} that the following results hold.
\begin{corollary}\label{corollary3.5}
Let $G=C_{T_1,\cdots,T_g}$ be an arbitrary unicyclic graph in $\mathcal{U}_{n,l}^{g,1}$.
Then for $p= \lfloor\frac{n-g-gl+l}{l+1}\rfloor$,
$$LEL(G)\geq \mbox{min}\{LEL(U^0), LEL(U^1), \cdots, LEL(U^p)\}.$$
\end{corollary}

\begin{corollary}\label{corollary3.6}
Let $G=C_{T_1,\cdots,T_g}$ be an arbitrary unicyclic graph in $\mathcal{U}_{n,l}^{g,2}$. Then
$LEL(G)> LEL(U^0)$.
\end{corollary}

\section{The minimal elements in $\mathcal{U}_{n,l}^3$ and $\mathcal{U}_{n,l}^4$}

In this section, we determine all the minimal elements in the posets $(\mathcal{U}_{n,l}^3, \preceq)$ and $(\mathcal{U}_{n,l}^4, \preceq)$.
Before stating our results, we need the following definitions.

\begin{definition}\label{definition4.1}
 For any $ G\in\mathcal{U}_{n,l}^3$, let $G^\ast$ be the graph obtained from $G$
by changing all the edges (except $E(C_3)$) incident with $u_2, u_3$ into new edges between
$u_1$ and $N_G(u_2)\cup N_G(u_3)\setminus V(C_3)$.
In other words,
\begin{eqnarray*}
G^\ast&=&G-\{u_2x| x\in N_G(u_2)\setminus V(C_3)\}-\{u_3x| x\in N_G(u_3)\setminus V(C_3)\}\\
&&+\{u_1x| x\in N_G(u_2)\cup N_G(u_3)\setminus V(C_3)\}.
\end{eqnarray*}
 We say $G^\ast$ is a $\eta$-transformation of $G$. (See Fig. 4).
\end{definition}
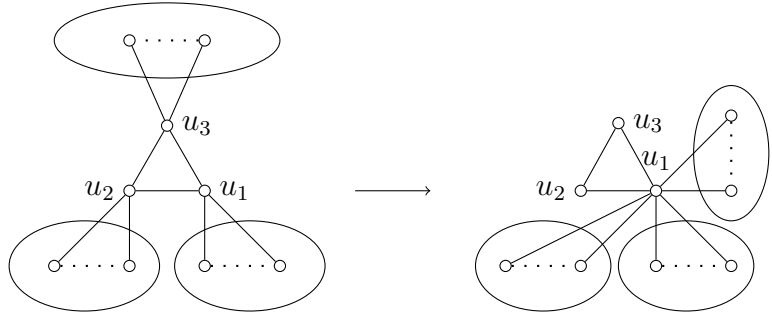
\begin{figure}\label{3}
\centering
\begin{tikzpicture}[scale=1]
\tikzstyle{every node}=[draw,shape=circle,inner sep=1.5pt];
\node (v0)[label=-180:$u_2$] at (0:0){};
\node(v1) [label=0:$u_1$]at (0:1){};
\node(v2) [label=0:$u_3$] at(60:1){};
\node(v3) at (1,2){};
\node(v4) at (0,2){};
\node(v5) at (-1,-1){};
\node(v6) at (0,-1){};
\node(v7) at (1,-1){};
\node(v8) at (2,-1){};
\draw (v0) --(v1)
(v1) --(v2)
(v2)--(v0) (v2)--(v4) (v2)--(v3)
(v0)--(v5) (v0)--(v6) (v1)--(v7) (v1)--(v8);
\draw [loosely dotted,thick] (v3) --(v4)
(v5)--(v6) (v7)--(v8);
\draw  (0.5,2) ellipse (1.5 and 0.5);
\draw (-0.6,-1) ellipse (1 and 0.6);
\draw (1.6,-1) ellipse( 1 and 0.6);
\tikzstyle{every node}=[draw,shape=circle,inner sep=1.5pt];
\node (v0)[label=-180:$u_2$] at (6,0){};
\node(v1) [label=88:$u_1$]at (7,0){};
\node(v2) [label=0:$u_3$] at(6.5,0.9){};
\node(v3) at (8,1){};
\node(v4) at (8,0){};
\node(v5) at (5,-1){};
\node(v6) at (6,-1){};
\node(v7) at (7,-1){};
\node(v8) at (8,-1){};
\draw (v0) --(v1)
(v1) --(v2)
(v2)--(v0) (v1)--(v4) (v1)--(v3)
(v1)--(v5) (v1)--(v6) (v1)--(v7) (v1)--(v8);
\draw [loosely dotted,thick] (v3) --(v4)
(v5)--(v6) (v7)--(v8);
\draw[->] (3,0)--(4,0);
\draw  (8,0.5) ellipse (0.5 and 0.9);
\draw (7.4,-1) ellipse (0.9 and 0.6);
\draw (5.5,-1) ellipse( 0.9 and 0.6);
\end{tikzpicture}
\caption{$\eta$-transformation} \label{fig:pepper}
\end{figure}

\begin{lemma}\label{lemma4.2}
For $G\in\mathcal{U}_{n,l}^3$, if $G^\ast$ is obtained from $G$ by $\eta$-transformation, then
$G^\ast\preceq G$, i.e., $c_{k}(G^\ast)\leq c_{k}(G)$ with equality
if and only if $k\in \{0,1,n-1,n\}$.
\end{lemma}

\begin{proof}
Clearly, when $k\in\{0,1,n-1,n\}$, $c_k(G)=c_k(G^\ast)$. For $2\leq k\leq n-2$,
let $\mathcal{F}^\prime$ (resp. $\mathcal{F}$) be the set of all spanning forests of $G^\ast$ (resp. $G$) with
exactly $n-k$ components.
Let $\mathcal{F}^\prime=\mathcal{F}^{\prime(1)}\cup \mathcal{F}^{\prime(2)}\cup\mathcal{F}^{\prime(3)}$,
where $\mathcal{F}^{\prime(j)}, j=1,2,3$ is the set of all spanning forests of $G^\ast$
in which $u_1, u_2, u_3$ belong to exactly $j$ different components. Similarly,
$\mathcal{F}^{(j)}, j=1,2,3$ can be defined. Let $f: \mathcal{F}^\prime\rightarrow \mathcal{F}$
with $F^\prime \rightarrow F=f(F^\prime)$, where $V(F)=V(F^\prime)$ and
\begin{align*}
E(F)=&E(F^\prime)\\
&-\{u_1x| x\in N_{R^\prime}(u_1)\cap N_G(u_2)\setminus \{u_3\}\}
-\{u_1x| x\in N_{R^\prime}(u_1)\cap N_G(u_3)\setminus \{u_2\}\}\\
&+\{u_2x| x\in N_{R^\prime}(u_1)\cap N_G(u_2)\setminus \{u_3\}\}
+\{u_3x| x\in N_{R^\prime}(u_1)\cap N_G(u_3)\setminus \{u_2\}\},
\end{align*}
for $R^\prime$ being a component of $F^\prime$ containing $u_1$. Clearly $f$ is injective and
$f(\mathcal{F}^{\prime(j)})\subseteq \mathcal{F}^{(j)}$ for $j=1,2,3$.
Denote $|V(T_1)\cap V(R^\prime)\setminus\{u_1\}|=a, |V(T_2)\cap V(R^\prime)\setminus\{u_2\}|=b,
|V(T_3)\cap V(R^\prime)\setminus\{u_3\}|=c$, where $a,b,c\geq 0$. Let $N$ be the product of
the orders of all components of $F^\prime$ containing no $\{u_1, u_2, u_3\}$.
Now we distinguish the proof into the following three cases.

{\bf Case 1:} $F^\prime\in \mathcal{F}^{\prime(1)}$, $u_1, u_2, u_3$ belong to one component,
then $\gamma(F)=\gamma(F^\prime)$. Thus
$$\sum_{F^\prime\in \mathcal{F}^{\prime(1)}}[\gamma(F)-\gamma(F^\prime)]=0.$$

{\bf Case 2:} $F^\prime\in \mathcal{F}^{\prime(2)}$, $u_1, u_2, u_3$ are in two components,
then $\gamma(F)-\gamma(F^\prime)=[(a+b+2)(c+1)+(a+c+2)(b+1)+(b+c+2)(a+1)-2(a+b+c+1)-2(a+b+c+2)]N=[(a+b)c+(a+c)b+(b+c)a]N\geq 0$. Thus
$$\sum_{F^\prime\in \mathcal{F}^{\prime(2)}}[\gamma(F)-\gamma(F^\prime)]\geq 0.$$

{\bf Case 3:} $F^\prime\in \mathcal{F}^{\prime(3)}$, $u_1, u_2, u_3$ are in three components,
then $\gamma(F)-\gamma(F^\prime)=[(a+1)(b+1)(c+1)-(a+b+c+1)]N=(abc+ab+ac+bc)N\geq 0$. Thus
$$\sum_{F^\prime\in \mathcal{F}^{\prime(3)}}[\gamma(F)-\gamma(F^\prime)]\geq 0.$$
Now the inequality $c_{k}(G^\ast)<c_{k}(G), k=2,3,\cdots, n-2$ holds from Theorem~\ref{theorem1.1}
by summing over all possible subsets $\mathcal{F}^\prime$ of spanning forests $F^\prime$ of $G^\ast$ with
$n-k$ components.
\end{proof}

\begin{theorem}\label{theorem4.3}
There are exactly $p+1$ minimal elements
$U_{n,l}^{3,0}, U_{n,l}^{3,1} \cdots, U_{n,l}^{3,p}$
in the  poset $(\mathcal{U}_{n,l}^{3},\preceq)$,
where $p= \lfloor\tfrac{n-3-2l}{l+1}\rfloor$.
    \end{theorem}
\begin{proof}
The assertion follows from Lemma~\ref{lemma4.2} and Theorems~\ref{theorem3.3} and \ref{theorem3.4}.
\end{proof}

\begin{definition}\label{definition4.4}
For $ G\in\mathcal{U}_{n,l}^{4}$, let $G^\star$ be the graph obtained from $G$
by changing all the edges (except $E(C_4)$) incident with $u_2, u_3, u_4$ into new edges between $u_1$ and
$\cup_{i=2}^4 N_G(u_i)\setminus V(C_4)$.
In other words,
\begin{eqnarray*}
G^\star&=&G-\{u_2x| x\in N_G(u_2)\setminus V(C_4)\}-\{u_3x| x\in N_G(u_3)\setminus V(C_4)\}\\
&&-\{u_4x| x\in N_G(u_4)\setminus V(C_4)\}+\{u_1x| x\in \cup_{i=2}^4 N_G(u_i)\setminus V(C_4)\}.
\end{eqnarray*}
We say $G^\star$ is a $\kappa$-transformation of $G$.
\end{definition}
Since for bipartite graphs, the Laplacian coefficients and the signless Laplacian coefficients are equal,
from Lemma 3.1 in \cite{zhang2012}, we have the following lemma:

\begin{lemma}\label{lemma4.5}
For $G\in\mathcal{U}_{n,l}^{4}$, if $G^\star$ is obtained from $G$ by $\kappa$-transformation, then
$G^\star\preceq G$, i.e., $c_{k}(G^\star)\leq c_{k}(G)$ with equality
if and only if $k\in \{0,1,n-1,n\}$.
\end{lemma}
By combining Lemma~\ref{lemma4.5} and Theorem~\ref{theorem3.3}, we have the following theorem.
\begin{theorem}\label{theorem4.6}
There are exactly $p+1$ minimal elements
$U_{n,l}^{4,0}, U_{n,l}^{4,1} \cdots, U_{n,l}^{4,p}$
in the poset $(\mathcal{U}_{n,l}^{4},\preceq)$,
where $p= \lfloor\tfrac{n-4-3l}{l+1}\rfloor$.
\end{theorem}

From Theorem~\ref{theorem1.2}, we have the following corollary:
\begin{corollary}\label{corollary4.7}
(1). Let $G=C_{T_1,T_2,T_3}\in \mathcal{U}_{n,l}^3$. Then for $p= \lfloor\tfrac{n-3-2l}{l+1}\rfloor$,
$$LEL(G)\geq \mbox{min}\{U_{n,l}^{3,0}, U_{n,l}^{3,1} \cdots, U_{n,l}^{3,p}\}.$$

(2). Let $G=C_{T_1,T_2,T_3,T_4}\in \mathcal{U}_{n,l}^4$. Then for $p= \lfloor\tfrac{n-4-3l}{l+1}\rfloor$,
$$LEL(G)\geq \mbox{min}\{U_{n,l}^{4,0}, U_{n,l}^{4,1} \cdots, U_{n,l}^{4,p}\}.$$
\end{corollary}

\section{Remarks}
Although Ili\'{c} and Ili\'{c}'s conjecture is false, we may modify the condition or result such that the conjecture is still true.
In fact, if there are at least two vertices in the cycle with degrees at least 3, then the conjecture is true for $g=3$ and $g=4
$. Moreover, we checked that the conjecture is still true for all unicyclic graphs on $\leq 30$ vertices
with fixed $l, g$ and having at least three vertices in the cycle with degrees at least 3. Hence their conjecture can be modified as follows:

\begin{conjecture}
(1).  For $G\in \mathcal{U}_{n,l}^g$, if there are more than two vertices in the cycle having degrees $\geq 3$, then
$U^0\preceq G$, with equality if and only if $G\cong U^0$.

(2). There are exactly $p+1$ minimal elements
$U_{n,l}^{3,0}, U_{n,l}^{3,1} \cdots, U_{n,l}^{3,p}$
in the  poset $(\mathcal{U}_{n,l},\preceq)$,
where $p= \lfloor\tfrac{n-3-2l}{l+1}\rfloor$.
\end{conjecture}

\end{document}